\documentclass[11pt,twoside]{amsart}

\usepackage{graphicx}

\newtheorem{thm}{Theorem}[section]
\newtheorem{prop}[thm]{Proposition}
\newtheorem{lem}[thm]{Lemma}

\theoremstyle{definition}
\newtheorem{defn}[thm]{Definition}

\newtheorem{rem}[thm]{Remark}

\theoremstyle{remark}
\newtheorem{claim}[thm]{Claim}

\newcommand{\pr}[1]{{\mathbb P}^{#1}}
\newcommand{\skipit}[1]{{}}
\newcommand{\prfend}{\hbox to7pt{\hfil}
\par\vskip-\baselineskip\hbox to\hsize
{\hfil\vbox {\hrule width6pt height6pt}}\vskip\baselineskip}

\newcommand{\Z}{\mathbb{Z}}
\newcommand{\R}{\mathbb{R}}
\newcommand{\N}{\mathbb{N}}
\newcommand{\Q}{\mathbb{Q}}
\newcommand{\TC}{\mathbb{T}}
\newcommand {\PP}{\mathbb{P}}
\newcommand {\GG}{\mathbb{G}}
\newcommand {\LL}{\mathbb{L}}
\newcommand {\FF}{\mathbb{F}}

\newcommand{\Oc}{\mathcal{O}}
\DeclareMathOperator{\cok}{cok}
\DeclareMathOperator{\Neg}{Neg}
\DeclareMathOperator{\nneg}{neg}
\DeclareMathOperator{\Cl}{Cl}
\DeclareMathOperator{\Image}{Im}
\DeclareMathOperator{\EFF}{EFF}
\DeclareMathOperator{\NEF}{NEF}
\newcommand {\C}[1]{\mathcal{#1}}
\DeclareMathOperator{\Pic}{Pic}
\DeclareMathOperator{\Char}{char}

\newcommand{\myarrow}[2]{\hbox to #1pt{\hfil$\to$\hfil}{\hskip-#1pt{\raise
10pt\hbox to#1pt{\hfil$\scriptscriptstyle #2$\hfil}}}}

\renewcommand{\theequation}{\thesection.\arabic{thm}} 

\begin{document}

\title{Hilbert functions of fat point subschemes of the plane: the Two-fold Way}

\author{A.\ V.\ Geramita, B.\ Harbourne \& J.\ Migliore}

\address{A.\ V.\ Geramita\\
Department of Mathematics\\
Queen's University\\
Kingston, Ontario\\
and Dipartimento di Matematica\\
Universit\`a di Genova\\
Genova, Italia}
\email{anthony.geramita@gmail.com}

\address{Brian Harbourne\\
Department of Mathematics\\
University of Nebraska\\
Lincoln, NE 68588-0130 USA}
\email{bharbour@math.unl.edu}

\address{Juan Migliore\\
Department of Mathematics\\
University of Notre Dame\\
South Bend, IN}
\email{migliore.1@nd.edu}

\date{January 25, 2011}

\markboth{A.\ V.\ Geramita, B.\ Harbourne, J.\ Migliore}{Classifying 
Hilbert functions of fat point subschemes in $\pr{2}$}

\thanks{Acknowledgments: We thank the NSF, whose 
support for the MAGIC05 conference at the University of Notre Dame gave
us an opportunity to begin work on this paper.
Geramita also thanks the NSERC for research support, and Harbourne and Migliore 
thank the NSA for its partial support of their research
(under grant H98230-07-1-0066 for Harbourne and under grants H98230-07-1-0036
and H98230-09-1-0031 for Migliore).}

\begin{abstract}
Two approaches for determining Hilbert functions
of fat point subschemes of $\pr2$ are demonstrated.
A complete determination of the Hilbert functions which occur for
9 double points is given using the first approach,
extending results obtained in a previous paper using the second 
approach. In addition the second approach is used to obtain a complete 
determination of the Hilbert functions
for $n\geq9$ $m$-multiple points for every $m$ if 
the points are smooth points of an irreducible plane cubic curve. 
Additional results are obtained using the first approach for $n\geq9$ double points
when the points lie on an irreducible cubic (but now are not assumed to be smooth points
of the cubic).
\end{abstract}

\maketitle

\section{Introduction}\label{intro}

If $X$ is a reduced set of $n$ points in $\pr2$, the {\em fat point subscheme $Z = mX\subset\pr2$} 
is the $(m-1)$-st infinitesimal neighborhood of $X$. Thus $mX$ is the subscheme defined by the 
symbolic power $I(X)^{(m)}\subset R=k[\pr2]$ (that is, by the saturation of the ideal $I(X)^m$
with respect to the ideal generated by the coordinate variables in the ring $k[\pr2]$).  
The question motivating this paper is: What 
are the Hilbert functions of such subschemes of $\pr2$?
There have been two main approaches to this question, 
and one goal of this paper is to demonstrate them in various situations.

The two approaches are exemplified by the papers \cite{GMS} and \cite{GHM}.
The approach of \cite{GMS} is to identify constraints that Hilbert functions
must satisfy and then for each function satisfying those constraints to try to find a 
specific subscheme having that function as its Hilbert function.
A complete classification of all Hilbert functions of reduced 0-dimensional
subschemes of projective space was given in \cite{GMR} using essentially this approach.
The paper \cite{GMS} then uses \cite{GMR} as the starting point for
classifying Hilbert functions for subschemes of the form $Z = 2X\subset\pr2$
with $X$ reduced and 0-dimensional.
This approach is most effective when the class of possible functions
is fairly limited, hence the restriction in \cite{GMS} to the case $m=2$.
This approach has the advantage of providing explicit results often without needing
detailed information about the disposition of the points, but it has the disadvantage
of not providing a complete dictionary of which point sets give which Hilbert function.
The approach of \cite{GHM} is to use the geometry of the surface
$Y$ obtained by blowing up the points of the support of $Z$ to
obtain information about the Hilbert function of $Z$.
This approach is most effective when the geometry of $Y$ is well-understood,
hence the restriction in \cite{GHM} to the case $n\leq 8$.
Given points $p_i$ and non-negative integers $m_i$,
the subscheme defined by the ideal $\cap_{i=1}^n(I(p_i)^{m_i})$
is also called a fat point subscheme, and is denoted
$m_1p_1+\cdots+m_np_n$.
The advantage of the second approach, as implemented in \cite{GHM},
is that it provided complete results for all fat point subschemes 
$Z=m_1p_1+\cdots+m_np_n$ with $n\leq8$,
together with a complete determination of which $Z$ give 
the same Hilbert function, but it had the cost of needing a lengthy
analysis of the geometry of $Y$, and gives only recursive
determinations of the Hilbert functions. However, for $n\leq8$ and
$k=2$ there are only finitely many cases, so a complete list of
the Hilbert functions which occur can be given. See \cite{GHM}
for this list. 

The first case left open by \cite{GHM} is $n=9$ points of $\pr2$.
It should, in principle, be possible to carry out the necessary analysis
to obtain a complete recursive classification of Hilbert functions and 
corresponding points sets for $n=9$, but whereas for $n\leq 8$ there are only finitely
many classes of sets of $n$ points, there will certainly be infinitely many when $n=9$
(related to the fact that there can be infinitely many prime divisors on $Y$ of 
negative self-intersection, and to the fact that effective nef divisors $F$ can occur with 
$h^1(Y, {\mathcal O}_Y(F))>0$). Thus a complete classification in this case
using the methods of \cite{GHM} will be a substantial effort, which we leave
for future research (not necessarily by us). 

Instead, in this paper we will focus on some special cases.
We devote \textsection\ref{9dblpts} to demonstrating the first approach by obtaining a complete answer in
the case of $n=9$ and $m=2$. This also shows how one could recover the result
for $n=8$ and $m=2$ obtained in \cite{GHM} using the methods of \cite{GMS}.

The rest of the paper is devoted to demonstrating both methods
for the case of $n$ points of multiplicity $m$ on cubics, under somewhat
different hypotheses chosen to play to the strengths of each method.
The Philosophy of the First Way is to use known facts about Hilbert functions to say things about what Hilbert functions are possible. The Philosophy of the Second Way is to use known facts about cohomology of blown up surfaces to say things about what dimensions of linear systems are possible. Sections 
\S\ref{ptsoncubics} (using the First Way) and \S\ref{appII} (using the Second Way)
illustrate how we can attack the same problem and obtain overlapping and sometimes complementary
results, but using dramatically different ways to do so.  

So, given points on a plane cubic,
for the First Way we will assume the cubic is irreducible, that $m=2$ 
and, in some cases, that $n$ is not too small.
Our main results here are Theorem \ref{smooth cubic} and Theorem \ref{singular cubic}.  
For the Second Way
we will make no restrictions on $m$ nor assume the cubic is irreducible but 
we will assume the points are smooth points of the cubic and 
we will assume that the points are evenly distributed (meaning essentially that
no component contains too many of the points).
Under these two assumptions we give a complete determination of 
all possible Hilbert functions in Theorem \ref{method2cor1}. 
Using the same techniques we will, in Remark \ref{method2cor2},
also recover the Hilbert functions 
for $X$ and $2X$ when $X$ is a reduced set of points contained in 
a reduced, irreducible singular cubic curve in case the singular 
point of the curve is one of the points of $X$.

We now discuss both methods in somewhat more detail.
For the first approach we will follow \cite{GHM} and \cite{GMS}
and sometimes work with the first difference, $\Delta h_{2X}$, of the Hilbert function $h_{2X}$
rather than with $h_{2X}$ directly, since for our purposes $\Delta h_{2X}$
is easier to work with, but we regard $\Delta h_{2X}$ as just an equivalent formulation 
of the Hilbert function and so for simplicity we will refer to it as the Hilbert function.  
The first approach can be summarized as follows.  We start by listing 
all Hilbert functions $\Delta h_X$ for reduced sets $X$ of $n$ points, using \cite{GMR},
and then we analyze each case in turn using $h_X$ to
constrain the behavior of $h_{2X}$. For example, in some extreme cases 
the form of $\Delta h_X$ forces many of the points of $X$ to lie on a line;
knowing this can be very useful in determining $h_{2X}$.

Our analysis uses the following tools:
(a) a crude bound on the regularity of $I(2X)$, giving an upper bound for the 
last degree in which $\Delta h_{2X}$ can be non-zero; 
(b) B\'ezout considerations giving the values of $\Delta h_{2X}$ in most degrees; 
(c) the fact that the sum of the values of $\Delta h_{2X}$ is 27; and
(d) a theorem of Davis \cite{davis} giving geometric consequences for certain 
behavior of the function $\Delta h_{2X}$.  The idea is that we know the value 
of the Hilbert function for most degrees by (a), (b) and (c), and we can 
exhaustively list the possibilities for the remaining degrees.  
Then we use (d) to rule out many of these.  Finally, for the cases that 
remain, we try to construct examples of them (and in the situations 
studied in this paper, we succeed).

For the second approach we study $h_Z$ for an arbitrary fat point subscheme
$Z=m_1p_1+\cdots +m_np_n\subset\pr2$ using the geometry
of the surface $Y$, where $\pi:Y\to\pr2$ is the morphism 
obtained by blowing up the points $p_i$. This depends on 
the well known fact that $\dim I(Z)_t=h^0(Y, {\mathcal O}_Y(F))$
where $F=tL-m_1E_1-\cdots-m_nE_n$, ${\mathcal O}_Y(L)=\pi^*{\mathcal O}_{\pr2}(1)$
and $E_i=\pi^{-1}(p_i)$. The fundamental fact here is 
the theorem of Riemann-Roch:
\addtocounter{thm}{1}
\begin{equation}\label{RRoch}
h^0(Y, {\mathcal O}_Y(F))-h^1(Y, {\mathcal O}_Y(F))+
h^2(Y, {\mathcal O}_Y(F))=\frac{F^2-K_Y\cdot F}{2}+1=\binom{t+2}{2}-\sum_i\binom{m_i+1}{2}.
\end{equation}
To see the relevance of \eqref{RRoch}, note that $K_Y=-3L+E_1+\cdots+E_n$, so we have by duality that 
$h^2(Y, {\mathcal O}_Y(F))=h^0(Y, {\mathcal O}_Y(K_Y-F))$ and
thus $h^2(Y, {\mathcal O}_Y(F))=0$ if $t<0$. Now, since we are interested in the values of Hilbert functions when
$t\geq0$, we have
$h_Z(t) = \dim(R_t)-\dim(I(Z)_t)=\binom{t+2}{2}-h^0(Y, {\mathcal O}_Y(F))$ which using \eqref{RRoch} becomes
\addtocounter{thm}{1}
\begin{equation}\label{RRoch2}
h_Z(t) =\sum_i\binom{m_i+1}{2}-h^1(Y, {\mathcal O}_Y(F)).
\end{equation}

This second approach, as applied in \cite{GHM}, 
depended on knowing two things: the set $\hbox{Neg}(Y)$
of all prime divisors $C$ on $Y$ with $C^2<0$ and on knowing
$h^0(Y,{\mathcal O}_Y(F))$ for every divisor $F$
for which we have $F\cdot C\geq0$ for all $C\in \hbox{Neg}(Y)$.
Given $\hbox{Neg}(Y)$,
one can in principle reduce the problem of computing 
$h^0(Y, {\mathcal O}_Y(F))$ for an arbitrary divisor $F$
to the case that $F\cdot C\geq0$ for all $C\in \hbox{Neg}(Y)$.
If $n\geq 2$ and $F\cdot C\geq 0$ for all $C\in \hbox{Neg}(Y)$,
then $h^2(Y, {\mathcal O}_Y(F))=0$, so from Riemann-Roch
we have only $h^0(Y, {\mathcal O}_Y(F))\geq 1+(F^2-K_Y\cdot F)/2$.

When $n\leq8$ or the points $p_i$ lie on a conic (possibly singular), 
this inequality is always an equality,
but for $n\geq9$ points not contained in a conic it needn't be, 
so more information in general is needed.
Similarly, in case $n\leq8$ or the points $p_i$ lie on a conic (possibly singular), 
it turns out, in fact, that $\hbox{Neg}(Y)$ is a finite set, but this also can fail 
for $n\geq9$ points not contained in a conic.
As a consequence, given $\hbox{Neg}(Y)$
one can determine $h_Z$
for any fat point subscheme $Z=m_1p_1+\cdots +m_np_n\subset\pr2$
if either $n\leq8$ or the points $p_i$ lie on a conic.
This raises the question of what sets $\hbox{Neg}(Y)$
occur under these assumptions. We answered this question
in \cite{GHM}. There are only finitely many possibilities
and \cite{GHM} gives a complete list.

When $n\geq 9$ and the points $p_i$
do not lie on a conic then not only can $\hbox{Neg}(Y)$ fail to be finite
but $h^1(Y, {\mathcal O}_Y(F))$ need not vanish, even if
$F\cdot C\geq 0$ for all $C\in \hbox{Neg}(Y)$ and even if $F$ is effective.
Assuming that the points $p_i$ lie on a cubic curve does not 
eliminate either difficulty, but it does mean that
$-K_Y$ is effective (whether the cubic is irreducible or not), 
and thus the results of \cite{anticanSurf}
can be applied to the problem of computing $h^0(Y, {\mathcal O}_Y(F))$.
In case $-K_Y$ is effective, it is known what kinds of classes
can be elements of $\hbox{Neg}(Y)$, but no one has yet classified precisely which sets $\hbox{Neg}(Y)$ arise for $n  
\geq 9$ (doing this for $n=7, 8$ was the new contribution in \cite{GHM}). On the other hand,
even without this complete classification, partial results can still be obtained
using the second approach, as we will show here 
using information about the geometry of $Y$ developed in \cite{anticanSurf}.


\section{Approach I: Nine Double Points}\label{9dblpts}

It is natural to ask what can be said for fat point schemes $Z$ supported at 
$r>8$ points.  As observed in \cite[Remark 2.2]{GHM}, there are infinitely
many configuration types of $r>8$ points, so we will restrict our attention
to subschemes $2Z=2(p_1+\cdots+p_r)$ of $\pr2$. Since we are now restricting the multiplicities of the points to be at most 2,
it is not necessary to make an exhaustive list of the configuration types -- indeed, we will point out situations where different configurations exist but nevertheless do not give different Hilbert functions. Instead, in this situation we can bring to bear
the methods developed in \cite{GMS}, and  to demonstrate additional methods
which can be used.  We will determine all Hilbert functions
that occur for double point subschemes $2Z=2(p_1+\cdots+p_9)$ of $\pr2$, for every Hilbert function occurring as the
Hilbert function of a simple point subscheme $Z=p_1+\cdots+p_9$. 

\begin{defn}\rm
Let $Z$ be a zero-dimensional subscheme of $\pr{n}$ with Hilbert function $h_Z$.  The {\em difference function} of $Z$ is the first difference of the Hilbert function of $Z$, $\Delta h_Z (t) = h_Z(t) - h_Z(t-1)$.  (This is sometimes also called the {\em $h$-vector} of $Z$, and sometimes the {\em Castelnuovo function} of $Z$.) 
\end{defn}

The Hilbert function and its difference function clearly give equivalent information and it is
primarily because of the simpler bookkeeping allowed by the first difference that we use it. 
Notice that $\Delta h_Z$ is the Hilbert function of any Artinian reduction of $R/I_Z$ by a linear form.   

One problem raised in \cite{GMS} is the existence and determination of maximal and minimal
Hilbert functions.  In the current context, this means that we fix an underlying Hilbert function $\underline{h}$ that exists for some set of 9 points in $\pr2$, and letting $X$ move in the irreducible flat family of all sets of points with Hilbert function $\underline{h}$, we ask whether there is a maximal and a minimal Hilbert function for the corresponding schemes $Z=2X$.  It was shown in \cite{GMS} that there {\em does} exist a maximal such Hilbert function, denoted ${\underline h}^{max}$ (for any number of points).  The proof in \cite{GMS} is nonconstructive, and \cite{GMS}
determines ${\underline h}^{max}$ in only a few special cases.
The paper \cite{GMS} also raises the question of whether 
${\underline h}^{min}$ always exists; i.e., whether there exists an $X'$
such that $h_{2X}$ is at least as big in every degree as $h_{2X'}$ for
every $X$ with $h_X=h_{X'}$. This question remains open.

A useful tool is the following lemma. This lemma, and generalizations of it, are well-known.
For a very short proof of the statement given here see \cite[Lemma 2.18]{GMS}.

\begin{lem} \label{reg lemma}
Let $X$ be a reduced set of points in $\pr2$ with regularity $r+1$.  Then the regularity of $I_{2X}$ is bounded by $\hbox{\rm reg}(I_{2X}) \leq 2 \cdot \hbox{\rm reg}(I_X) = 2r+2$.
\end{lem}

We will also use the following result of Davis \cite{davis}.  It is a special case of a more general phenomenon \cite{BGM} related to maximal growth of the first difference of the Hilbert function.

\begin{thm} \label{davis thm}
Let $X \subset \pr2$ be a zero-dimensional subscheme, and assume that $\Delta h_X (t) = \linebreak \Delta h_X(t+1) = d$ for some $t,d$.  Then the degree $t$ and the degree $t+1$ components of $I_X$ have a GCD, $F$, of degree $d$.  Furthermore, the subscheme $W_1$ of $X$ lying on the curve defined by $F$ (i.e.\ $I_{W_1}$ is the saturation of the ideal $(I_X,F)$) has Hilbert function whose first difference is given by the truncation
\[
\Delta h_{W_1} (s) = \min \{ \Delta h_X(s), d \}.
\]
Furthermore, the Hilbert function of the points $W_2$ not on $F$ (defined by $I_{W_2} = I_X : (F)$) has first difference given by the (shifted) part {\em above} the truncation:
\[
\Delta h_{W_2}(s) = \max \{ \Delta h_X (s+d) -d , 0 \}.
\]
\end{thm}

We will see precisely the possibilities that occur for the first infinitesimal neighborhood of nine points, and we will see that there is in each case a maximum and minimum
Hilbert function. All together, there occur
eight Hilbert functions for schemes $X=p_1+\cdots+p_9$.
We give their difference functions, and the possible Hilbert functions that occur for double point schemes
$2X$, in the following theorem.

\begin{thm}\label{9 pts}
The following table lists all possibilities for the difference function for nine double points, in terms of the difference function of the underlying nine points.  In particular, for each ${\underline h}$, both ${\underline h}^{max}$ and ${\underline h}^{min}$ exist, and we indicate by ``max'' or ``min" the function that achieves ${\underline h}^{max}$ or ${\underline h}^{min}$, respectively, for each $\underline{h}$.  Of course when we have ``max = min," the Hilbert function of $2X$ is uniquely determined by that of $X$.

\addtocounter{thm}{1}
\begin{equation} \label{hfs of nine points}
\begin{array}{l|l|lccccccccccccccccccccccccccccccccc}
\hbox{\rm difference function of $X$} & \hbox{\rm possible difference functions of $2X$} & \hbox{\rm max/min} \\ \hline
\verb! 1 1 1 1 1 1 1 1 1! & \verb! 1 2 2 2 2 2 2 2 2 2 1 1 1 1 1 1 1 1! & \hbox{\rm max = min} \\ \hline
\verb! 1 2 1 1 1 1 1 1! & \verb! 1 2 3 4 2 2 2 2 2 1 1 1 1 1 1 1! & \hbox{\rm max = min} \\ \hline
\verb! 1 2 2 1 1 1 1! & \verb! 1 2 3 4 4 3 2 2 1 1 1 1 1 1! & \hbox{\rm max = min} \\ \hline
\verb! 1 2 2 2 1 1! & \verb! 1 2 3 4 4 4 3 2 1 1 1 1! & \hbox{\rm max = min} \\ \hline 
\verb! 1 2 2 2 2! & \verb! 1 2 3 4 4 4 4 2 2 1!  & \hbox{\rm max} \\
& \verb! 1 2 3 4 4 4 3 2 2 2! & \hbox{\rm min} \\ \hline
\verb! 1 2 3 1 1 1! & \verb! 1 2 3 4 5 5 2 1 1 1 1 1! & \hbox{\rm max = min} \\ \hline
\verb! 1 2 3 2 1! & \verb! 1 2 3 4 5 6 4 2!  & \hbox{\rm max} \\
& \verb! 1 2 3 4 5 6 3 2 1! \\
& \verb! 1 2 3 4 5 6 3 1 1 1! \\
& \verb! 1 2 3 4 5 6 2 2 1 1! & \hbox{\rm min}
\\ \hline
\verb! 1 2 3 3! & \verb! 1 2 3 4 5 6 6! & \hbox{\rm max} \\
& \verb! 1 2 3 4 5 6 5 1! \\
& \verb! 1 2 3 4 5 6 4 2! \\
& \verb! 1 2 3 4 5 6 3 3! \\
& \verb! 1 2 3 4 5 5 4 3! & \hbox{\rm min}
\end{array}
\end{equation}

\end{thm}

\begin{proof} 
One has to ``integrate" the difference functions in order to verify the claims about ${\underline h}^{max}$ or ${\underline h}^{min}$.  We leave this to the reader.  The fact that the eight Hilbert functions listed above for $X$ give a complete list is standard, and we omit the proof.

\bigskip

\noindent \underline{Case 1}: \verb! 1 1 1 1 1 1 1 1 1!.  If $X$ has this  difference function then $X$ must be a set of 9 collinear points in $\pr2$.  Such a set of points is necessarily a complete intersection, so it is easy to check that the difference function for $2X$ is the one claimed.  (Even the minimal free resolution is well-known.)

\bigskip

\noindent \underline{Case 2}: \verb! 1 2 1 1 1 1 1 1!.  If $X$ has this difference function then $X$ must consist of 8 points on a line and one point off the line (it follows from Theorem \ref{davis thm}).  It is not hard to check, using B\'ezout arguments, that then $2X$ has the claimed difference function.

\bigskip

\noindent \underline{Case 3}:  \verb! 1 2 2 1 1 1 1!.  If $X$ has this difference function then $X$ must consist of seven points on a line, say $\lambda_1$, and two points off the line (again using Theorem \ref{davis thm}).  Let $Q_1, Q_2$ be these latter points.  We will see that the Hilbert function is independent of whether  $Q_1$ and $Q_2$ are collinear with one of the seven other points or not.  Note first that $2X$ contains a subscheme of degree 14 lying on a line.  Hence the regularity is $\geq 14$, so the difference function ends in degree $\geq 13$.  

Let $L_1$ be a linear form defining $\lambda_1$ and let $L_2$ be a linear form defining the line joining $Q_1$ and $Q_2$.  Using B\'ezout's theorem, it is clear that there is no form of degree $\leq 3$ vanishing on $2X$.  Furthermore, $L_1^2 L_2^2$ is the only form (up to scalar multiples) of degree 4 vanishing on $2X$.  Now, in degree 5 we have that $L_1^2$ is a common factor for all forms in the ideal of $2X$.  Hence $ (I_{2Q_1 + 2Q_2})_3 \cong (I_{2X})_5 $, where the isomorphism is obtained by multiplying by $L_1^2$.  But $2Q_1 + 2Q_2$ imposes independent conditions on forms of degree 3, so these isomorphic components have dimension $10 - 6 = 4$.  

The calculations above give the claimed difference function up to degree 5.  But the sum of the terms of the difference function has to equal 27 ($= \deg 2X$), and the terms past degree 5 must be non-increasing and positive and non-zero through degree 13.  Using also Lemma \ref{reg lemma} (which implies that the difference function must be zero no later than degree 14), this is enough to force the claimed difference function.

\bigskip

\noindent \underline{Case 4}: \verb! 1 2 2 2 1 1!.  By Theorem \ref{davis thm}, $X$ must consist of six points, $X_1$, on a line, $\lambda_1$,  and three collinear points, $X_2$, on another line, $\lambda_2$.  The intersection of $\lambda_1$ and $\lambda_2$ may or may not be a point of $X_1$; it is not a point of $X_2$.  We will see, as in Case 3, that this combinatorial distinction does not affect the Hilbert function of $2X$.  Pictorially we have the following two possibilities:

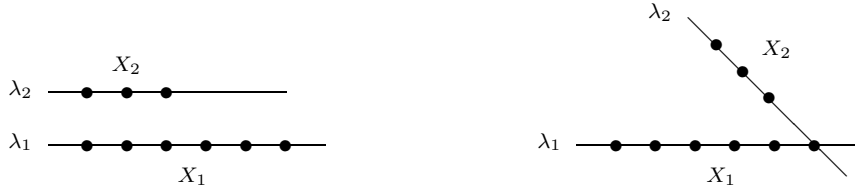
\begin{picture}(200,90)
\put (80,18){\line (1,0){105}}
\put (65,17){$\scriptstyle \lambda_1$}
\put (92,15){$\bullet$}
\put (107,15){$\bullet$}
\put (122,15){$\bullet$}
\put (137,15){$\bullet$}
\put (152,15){$\bullet$}
\put (167,15){$\bullet$}
\put (80,38){\line (1,0){90}}
\put (65,37){$\scriptstyle \lambda_2$}
\put (92,35){$\bullet$}
\put (107,35){$\bullet$}
\put (122,35){$\bullet$}
\put (104,46){$\scriptstyle X_2$}
\put (129,4){$\scriptstyle X_1$}

\put (280,18){\line (1,0){105}}
\put (265,17){$\scriptstyle \lambda_1$}
\put (292,15){$\bullet$}
\put (307,15){$\bullet$}
\put (322,15){$\bullet$}
\put (337,15){$\bullet$}
\put (352,15){$\bullet$}
\put (367,15){$\bullet$}
\put (322,66){\line(1,-1){60}}
\put (330,53){$\bullet$}
\put (340,43){$\bullet$}
\put (350,33){$\bullet$}
\put (307,66){$\scriptstyle \lambda_2$}
\put (350,52){$\scriptstyle X_2$}
\put (329,4){$\scriptstyle X_1$}
\end{picture}

Combining  Lemma \ref{reg lemma} with the fact that $2X$ contains a subscheme of degree 12 on a line, we get that the difference function of $2X$ ends in degree exactly 11.  Using B\'ezout it is not hard to check  that 
\[
\begin{array}{lll}
h^0({\mathcal I}_{2X}) = h^0({\mathcal I}_{2X}(1)) = h^0({\mathcal I}_{2X}(2)) = h^0({\mathcal I}_{2X}(3)) = 0 \\
h^0({\mathcal I}_{2X}(4)) = 1 \\
h^0({\mathcal I}_{2X}(5)) = h^0({\mathcal I}_{2X_2}(3)) = h^0({\mathcal I}_{X_2}(2)) = 3 \\
h^0({\mathcal I}_{2X}(6)) = h^0({\mathcal I}_{2X_2}(4)) = h^0({\mathcal I}_{X_2}(3)) = 7.
\end{array}
\]
This means that the difference function of $2X$ begins \verb!1 2 3 4 4 4 3 ...! and arguing as in Case 3 gives the result.

\bigskip

\noindent \underline{Case 5}: \verb! 1 2 2 2 2!.  This case corresponds to nine points on a reduced conic curve.  There are three possibilities.  If the conic is smooth then the nine points are arbitrary.  If the conic consists of two lines then this case takes the form of five points on one line and four points on the other line.  Here we can have (a) none of the nine points is the point of intersection of the two lines, or (b) one of the five points is the point of intersection.  All of these cases have been studied in \cite{GHM}, and we omit the details.

\bigskip

\noindent \underline{Case 6}: \verb! 1 2 3 1 1 1!.  Now $X$ consists of six point on a line plus three non-collinear points off the line.  It is easy to check, using the same methods, that there is only one possibility for the Hilbert function of $2X$, independent of whether the line through two of the non-collinear points meets one of the six collinear points or not.  We omit the details.

\bigskip

\noindent \underline{Case 7}: \verb! 1 2 3 2 1!.  By  Lemma \ref{reg lemma}, the difference function for $2X$ ends in degree $\leq 9$ and the entries again add up to 27.  Furthermore, it is not hard to see that $X$ has at most 5 points on a line, and $X$ has at most one set of 5 collinear points.  

The first main step in the proof is the following assertion:

\medskip

 
 \begin{claim} \label{old claim 7.1}
 $h^0({\mathcal I}_{2X}(5)) = 0$.  
 \end{claim}

Note that this implies that $h^0({\mathcal I}_{2X}(t)) = 0$ for $t \leq 5$.  Suppose that there is a curve $F$, of degree 5 containing $2X$.  There are several possibilities.  By abuse of notation we will denote by $F$ also a form defining this curve.

\begin{itemize}
\item $F$ is reduced.  Then $F$ has to contain 9 singular points, which form the points of $X$ (and hence have the difference function \verb!1 2 3 2 1!).  This can happen in one of two ways:

\begin{itemize}
\item $F$ consists of the union of five lines, and $X$ consists of nine of the resulting ten double points.  But from B\'ezout we note that the 10 double points do not lie on a cubic curve (since each of the five  lines would have to be a component of the cubic), so the ten points have difference function \verb!1 2 3 4!, and hence $X$ cannot have difference function \verb!1 2 3 2 1!.

\item $F$ consists of the union of three lines and a smooth conic, and $X$ consists of all nine resulting double points.  Now the three lines have to be components of any cubic containing $X$, so there is a unique such cubic, and again $X$ does not have difference function \verb!1 2 3 2 1!.
\end{itemize}

\item $F$ has a double conic.  Then all the singular points of $F$ must lie on this conic.  But, $X$ does not lie on a conic, so this is impossible.

\item $F$ has a double line, i.e. $F = L^2G$, $\deg G = 3$.  Then $G$ contians at most 3 singular points of $F$.   This forces the remaining 6 singular points to be on the line, contradicting the fact that at most 5 points of $X$ can lie on a line.
\end{itemize}

\noindent This concludes the proof of Claim \ref{old claim 7.1}.

\medskip

Thanks to Claim \ref{old claim 7.1}, we now know that the difference function for $2X$ has the form
\[
\verb! 1  2  3  4  5  6! \qed \qed \qed \qed
\]
where the last four spaces correspond to entries that are $\geq 0$ and add up to $27 - 21 = 6$.  
Now notice that there is an irreducible flat family of subschemes of degree 9 with difference function \verb!1 2 3 2 1!   \cite{ellingsrud}, and the general such is a complete intersection of two cubics.  The difference function for the corresponding scheme $2X$ is easily checked to be  \verb!1 2 3 4 5 6 4 2!.  It follows that not only does this difference function exist, but in fact it corresponds to ${\underline h}^{max}$.  (See also \cite[Remark 7.4]{GMS}.)
In particular, \verb!1 2 3 4 5 6 6! and \verb!1 2 3 4 5 6 5 1! do not occur.  The following, then, are the remaining possibilities for the difference function of $2X$:

\begin{enumerate}
\item \verb! 1 2 3 4 5 6 4 2!

\item \verb! 1 2 3 4 5 6 4 1 1!

\item \verb! 1 2 3 4 5 6 3 3!

\item \verb! 1 2 3 4 5 6 3 2 1!

\item \verb! 1 2 3 4 5 6 3 1 1 1!

\item \verb! 1 2 3 4 5 6 2 2 2!

\item \verb! 1 2 3 4 5 6 2 2 1 1!
\end{enumerate}

\noindent For each of these we will either give a specific example (that the reader can verify directly, either by hand or on a computer program) or a proof of non-existence.

\begin{enumerate}
\item \verb! 1 2 3 4 5 6 4 2! .  As we saw above, this occurs when $X$ is the complete intersection of two cubics, and this corresponds to ${\underline h}^{max}$.

\item \verb! 1 2 3 4 5 6 4 1 1!.  This does not exist.  Indeed, this difference function forces the existence of a line $\lambda$ that contains a subscheme of $2X$ of degree 9, which is impossible.  (Any such subscheme must have even degree.)

\item \verb! 1 2 3 4 5 6 3 3!.  This does not exist in our context.  Note that it {\em does} exist when $X$ has difference function \verb!1 2 3 3!, as we will verify below.
To see that this does not exist, note that by Theorem \ref{davis thm},  the \verb!3 3! at the end forces the existence of a cubic curve $C$ that cuts out from $2X$ a subscheme $W$ of degree 21 with difference function \verb!1 2 3 3 3 3 3 3!.  Observe that if $P$ is a point of $X$ which is a smooth point of $C$, then $C$ cuts out a non-reduced point of degree 2 at $P$.  If $P$ is a point of $X$ which is a singular point of $C$, then $C$ contains the fat point $2P$ (which has degree 3). 
Note also that our $h$-vector does not permit the existence of a subscheme of degree more 
than 8 on a line. 

Suppose first that $C$ is reduced.  Since we only have the nine points of $X$ to work with, it is not hard to check, using the above observation, that the only way that $C$ can cut out from $2X$ a subscheme of degree 21 is if $X$ has the following configuration:

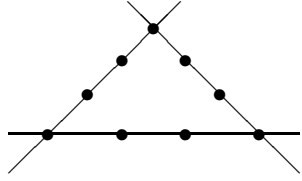
\begin{picture}(200,90)
\put (130,0){\line (1,1){65}}
\put (240,0){\line (-1,1){65}}
\put (130,15){\line(1,0){110}}
\put (142,12){$\bullet$}
\put (170,12){$\bullet$}
\put (194,12){$\bullet$}
\put (222,12){$\bullet$}
\put (157,27){$\bullet$}
\put (170,40){$\bullet$}
\put (182,52){$\bullet$}
\put (194,40){$\bullet$}
\put (207,27){$\bullet$}
\end{picture}

\bigskip

\noindent But this uses all nine points, and its support lies on a {\em unique} cubic, contradicting the fact that $X$ has difference function \verb!1 2 3 2 1!.  This configuration provides one of the correct difference functions for \verb!1 2 3 3! below.

Now suppose that $C$ is not reduced.  Without loss of generality, $C$ has a double line.  The difference function for $X$ would, in principle, allow five points of $X$ to lie on a line, but because the hypothetical difference function for the subscheme $W$ ends in degree 7, in fact at most four points of $X$ can lie on a line.  So the double line contains at most four fat points of $2X$, which have degree 12.  In order for $C$ to cut out a subscheme of degree 21, then, we must have a reduced line that cuts out an additional subscheme of degree at least 9.  This forces at least five points of $X$ to be collinear, which again is impossible.

\item \verb! 1 2 3 4 5 6 3 2 1! .  This difference function does exist.  It occurs when $X$ is the union of one point and the complete intersection of a conic and a general quartic curve.  

\item \verb! 1 2 3 4 5 6 3 1 1 1! . This difference function does exist.  It occurs when $X$ is the union of five general points on a line, three general points on a second line, and one additional general point off both lines.

\item \verb! 1 2 3 4 5 6 2 2 2!.  This difference function does not exist.  Indeed, suppose that it did exist.  Because of the \verb!2 2 2!, there must be a curve $C$ of degree 2 that cuts out on $2X$ a subscheme $W$ of degree 17 having difference function \verb!1 2 2 2 2 2 2 2 2!.  

First note that $X$ cannot contain five points on a line (and hence a subscheme of $W$ of degree at least 10) since the hypothetical difference function ends in degree 8. Now consider cases.

\begin{enumerate}
\item $C$ is smooth: then it cannot cut out a subscheme of odd degree.

\item $C$ is reduced and reducible: then we cannot obtain the desired subscheme $W$ of degree 17 unless $X$ contains 5 points on a line, in which case $W$ contains a subscheme of degree at least 10 on that line.

\item $C$ non-reduced: then we cannot have a subscheme of degree 17 supported on that line.
\end{enumerate}

\item \verb! 1 2 3 4 5 6 2 2 1 1!.  This difference function does exist.  It occurs when $X$ has the following configuration:

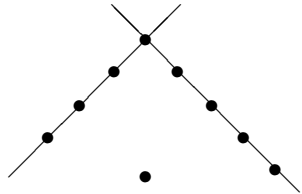
\begin{picture}(200,90)
\put (133,3){\line (1,1){65}}
\put (243,-3){\line (-1,1){71}}
\put (145,15){$\bullet$}
\put (182,0){$\bullet$}
\put (219,15){$\bullet$}
\put (157,27){$\bullet$}
\put (170,40){$\bullet$}
\put (182,52){$\bullet$}
\put (194,40){$\bullet$}
\put (207,27){$\bullet$}
\put (231,3){$\bullet$}
\end{picture}
\end{enumerate}

\bigskip

\noindent \underline{Case 8}: \verb! 1 2 3 3!. This is the difference function for a general set of nine points in $\pr2$.  We know (from \cite{anticanSurf}, for example) that the ``generic" difference function for nine general double points is \verb!1 2 3 4 5 6 6!.  Hence this occurs and corresponds to  the maximum possible Hilbert  function.  Clearly all other possibilities will end in degree $\geq 7$.  On the other hand,  Lemma \ref{reg lemma} guarantees that all other examples end in degree $\leq 7$.  Note that again, $X$ can have at most four points on a line.

\medskip

\underline{Claim 8.1}: $h^0 ({\mathcal I}_{2X}(5)) \leq 1$.

Notice that as a consequence of this claim we also obtain $h^0({\mathcal I}_{2X}(4)) = 0$.  Keeping in mind that it is possible that $h^0({\mathcal I}_{2X}(5)) = 0$ (e.g. the generic case), we will assume that $h^0({\mathcal I}_{2X}(5)) \neq 0$ and deduce that then it must be $=1$.  So let $C$ be a curve of degree 5 containing the scheme $2X$.  As before (Claim \ref{old claim 7.1}) there are a few possibilities.

\begin{itemize}
\item If $C$ is reduced then since it must have nine double points, it must consist of either the union of five lines, no three through a point, or the union of three lines and a smooth conic, with no three components meeting in a point.  By B\'ezout, each component of $C$ is then a fixed component of the linear system $|(I_{2X})_5|$, so the claim follows.

\item If $C$ contains a double line then at most four (fat) points of $2X$ lie on this line, so we must have a cubic curve that contains the remaining five double points.  Consider the support, $X_1$, of these five double points.  The points of $X_1$ are not collinear, and they do not have four collinear points since $X$ lies on only one cubic.  With these restrictions, clearly there is no cubic curve double at such a set of five points.

\item If $C$ contains a double conic (smooth or not), this conic contains at most seven points of $X$, because of the Hilbert function of $X$.  Hence $C$ must have a line that contains two double points, which is impossible.
\end{itemize}

This concludes the proof of Claim 8.1.

\medskip

It follows that the possibilities for the difference function of $2X$ are the following:

\begin{enumerate}
\item \verb! 1 2 3 4 5 6 6!

\item \verb! 1 2 3 4 5 6 5 1!

\item \verb! 1 2 3 4 5 6 4 2!

\item \verb! 1 2 3 4 5 6 3 3!

\item \verb! 1 2 3 4 5 5 5 2!

\item \verb! 1 2 3 4 5 5 4 3!
\end{enumerate}

\noindent As before, we examine these each in turn.

\begin{enumerate}
\item \verb! 1 2 3 4 5 6 6!. We have seen that this occurs generically.

\bigskip

\item \verb! 1 2 3 4 5 6 5 1!.  This exists, for instance from the following configuration: 



\noindent (That is, seven points on a conic, three points on a line, with one point in common.)

\bigskip

\item \verb! 1 2 3 4 5 6 4 2!.  This exists, for instance from the following configuration:

%

\item \verb! 1 2 3 4 5 6 3 3!. This exists, for instance from the configuration mentioned earlier:

%

\bigskip

\item \verb! 1 2 3 4 5 5 5 2!  We claim that this does not exist.  The key is that such a double point scheme, $2X$, would have to lie on a unique quintic curve, say $C$.  To see that this is impossible, the argument is very similar to that of Claim \ref{old claim 7.1}, but with a small difference.  One checks as before that $C$ must consist either of five lines or the union of three lines and a conic, and in both cases we must have that no three components share a common point.  In the first case, $X$ consists of nine of the ten double points of $C$ (it does not matter which nine), and in the second case $X$ consists of all nine double points of $C$.  But in both of these cases one can check geometrically or on a computer that $h^0({\mathcal I}_{2X} (6)) = 4$, while the hypothetical difference function would require this dimension to be 3.

\bigskip

\item \verb! 1 2 3 4 5 5 4 3!  This exists, and can be achieved by the configuration mentioned above: it is supported on nine of the ten intersection points of five general lines in $\pr2$.

\end{enumerate}
\end{proof}

\section{Approach I: Points on Cubics}\label{ptsoncubics}

For this section we will always let $C \subset \pr2$ be an irreducible cubic curve
defined by a polynomial $F$ of degree 3. Let $X$ be a 
reduced set of $n = 3t + \delta$ points on $C$, where 
$0 \leq \delta \leq 2$.   Let $Z= 2X$ be the double point scheme 
in $\pr2$ supported on $X$.  
The object of this section is to describe the possible Hilbert functions of 
$X$ and of the corresponding $Z$.  In some instances we assume that $t$ 
is ``big enough" (with mild bounds), and in one instance (Theorem \ref{smooth cubic}(b)) 
we assume that the points are not too special and that $C$ is smooth.

\bigskip

\begin{prop} \label{ci case}  Assume that $\delta = 0$, $t \geq 3$, and the Hilbert function of $X$ has first difference 
\addtocounter{thm}{1}
\begin{equation} \label{hf  of 1 2 3 ... 3 2 1}
\begin{array}{l|ccccccccccccccccccccccccccccccccccc}
\deg & 0 & 1 & 2  & \dots & t -1 & t & t+1 & t+2   \\
\hline
\Delta h_X & 1 & 2 & 3  & \dots & 3 & 2 & 1 & 0 
\end{array}
\end{equation}
(where the values between $2$ and $t-1$, if any, are all 3).  Then  $X$ is a complete intersection with ideal $(F,G)$, where 
$\deg F = 3$ and $\deg G = t$.  Furthermore, if $C$ is singular 
then the singular point is not a point of $X$.  Assume that $t>3$, 
so that $t+3 > 6$ and $2t > t+3$. Then we have the first difference of the Hilbert function of  $Z$ is

\medskip

\begin{itemize}

\item[($t=3$)]  $1 \ 2 \ 3 \ 4 \ 5 \ 6 \ 4 \ 2 \ 0$;

\item[($t=4$)]  $1 \ 2 \ 3 \ 4 \ 5 \ 6 \ 6 \ 5 \ 3 \ 1 \ 0$; 

\item[($t=5$)] $1 \ 2 \ 3 \ 4 \ 5 \ 6 \ 6 \ 6 \ 5 \ 4 \ 2 \ 1 \ 0$; 

\item[($t \geq 6$)] 
{\small 
\[
\begin{array}{l|ccccccccccccccccccccccccccccccccccc}
\deg & 0 & 1 & 2 & 3 & 4 & 5 & 6 & \dots & t +2 & t+3 & t+4 & t+5 & \dots & 2t-1 & 2t & 2t+1 & 2t+2   \\
\hline
\Delta h_Z & 1 & 2 & 3 & 4 & 5 & 6 & 6 & \dots & 6 & 5 & 4 & 3 & \dots & 3 & 2 & 1 & 0 
\end{array}
\]
}
\end{itemize}

\end{prop}

\begin{proof}
We first show that $X$ must be a complete intersection.  From the 
Hilbert function (\ref{hf of 1 2 3 ... 3 2 1}), it is clear that $F$ is a 
factor of every form in $I_X$ up to degree $t-1$, and that in fact 
it generates the ideal up to this point.  In degree $t$ there is 
exactly one  new form, $G$, in the ideal, and since $F$ is 
irreducible, $F$ and $G$ have no common factor.  But $(F,G)$ 
is a saturated ideal that is contained in $I_X$ and defines a 
zero-dimensional scheme of the same degree as $X$, hence $I_X = (F,G)$.

Since $X$ is a complete intersection, if $C$ is singular and 
$P \in X$ is the singular point of $C$, then $X$ must be 
non-reduced at $P$, contradicting our assumption.  

Now, it is a simple (and standard) argument that $I_Z = (F^2,FG,G^2)$, 
and one can verify  the claimed  Hilbert function of $R/I_Z$, 
for instance by using the fact that $(F,G)$ is directly linked 
to the ideal of $Z$ by the complete intersection $(F^2,G^2)$, and using the formula for the behavior of Hilbert functions under linkage \cite{DGO} (see also \cite{migliore}).  We omit the details.
\end{proof}

\bigskip

Because the form $F$ of least degree is irreducible, the Hilbert function 
of $X$ has first difference that is strictly decreasing from the first 
degree where it has value $< 3$ until it reaches 0.  Having 
proved Proposition \ref{ci case}, we can now assume without 
loss of generality that the Hilbert function of $X$ has first difference
\addtocounter{thm}{1}
\begin{equation} \label{hf of 1 2 3 ... 3}
\begin{array}{l|ccccccccccccccccccccccccccccccccccc}
\deg & 0 & 1 & 2 & 3 & \dots & t & t+1 & t+2   \\
\hline
\Delta h_X & 1 & 2 & 3 & 3 & \dots & 3 & \delta & 0
\end{array}
\end{equation}
where $0 \leq \delta \leq 2$.

\begin{thm} \label{smooth cubic}
Assume that either $C$ is smooth, or else that no point of $X$ is 
the singular point of $C$.  Assume further that $t > 5-\delta$.  
Then the Hilbert function of the double point scheme $Z$ supported on $X$ is
{\small 
\[
\begin{array}{l|ccccccccccccccccccccccccccccccccccc}
\deg & 0 & 1 & 2 & 3 & 4 & 5  & \dots & t+3 & t+4 & t+5 & \dots & 2t+\delta-1 & 2t+\delta   \\
\hline
\Delta h_Z & 1 & 2 & 3 & 4 & 5 & 6 &  \dots & 6 & 3+\delta & 3 & \dots & 3 & ??
\end{array}
\]
}
For the behavior in degree $\geq 2t+\delta$, we have the following conclusions.

\begin{itemize}
\item[(a)] If $\delta = 1$ or $\delta = 2$ then $\Delta h_Z (2t+\delta) = 3-\delta$ and $\Delta h_Z(k) = 0$ for $k > 2t+\delta$.

\item[(b)] If $\delta = 0$, there are two possible Hilbert functions, these being determined by 

\begin{itemize}
\item[i.] $\Delta h_Z(2t) = 3$ and $\Delta h_Z(k) = 0$ for $k > 2t$, and 

\item[ii.] $\Delta h_Z(2t) = 2, \Delta h_Z(2t+1) = 1, \Delta h_Z(2t+2) = 0$. 

\end{itemize}

\noindent  Moreover, if the points $p_i$ are sufficiently general and $C$ is smooth, then the Hilbert function is the first of these two.

\end{itemize}

\end{thm}

\begin{proof}
A complete analysis of all cases with $\delta = 0$, where $C$ is a 
reduced cubic and the points $p_i$ either are arbitrary smooth points 
of $C$ or they are completely arbitrary and $C$ is also irreducible, 
is given in the next section using the Second Way.  The interested 
reader can complete the current proof to those cases using the 
techniques of this section, as a further comparison of the methods.

First note that the condition $t > 5-\delta$ implies $2t+\delta > t+5$.  
We proceed via a number of claims.

\bigskip

\noindent {\bf Claim 1:}  {\em For $\ell < 2t+\delta$, $(I_Z)_\ell$ has 
the cubic form $F$ as a common factor (i.e.\ $C$ is part of the base locus).}

Suppose that $G \in (I_Z)_\ell$ does not have $F$ as a factor.  
Then at each point of $X$, the intersection multiplicity of $F$ and 
$G$ is at least 2 since $G$ is double at each point.  Hence by 
B\'ezout's theorem, $3\ell \geq 2n = 2(3t+\delta) = 6t+2\delta$.  
Hence $\ell \geq 2t + \frac{2}{3} \delta$, and the claim follows.

\bigskip

\noindent {\bf Claim 2:} {\em For $\ell \leq t+3$, $(I_Z)_\ell$ has $F^2$ as a common factor.}

By Claim 1, since $F$ is not double at any point of $X$, for $\ell < 2t+\delta$ we have an isomorphism
\addtocounter{thm}{1}
\begin{equation} \label{isom of ideals}
(I_X)_{\ell -3} \cong (I_Z)_\ell
\end{equation}
where the isomorphism is given by multiplication by $F$.  But from (\ref{hf of 1 2 3 ... 3}), 
we see that $F$ is a common factor for $(I_X)_k$ for all $k \leq t$.  
Hence $(I_Z)_\ell$ has $F^2$ as a factor whenever $\ell -3 \leq t$, as claimed.

This verifies the claimed first difference of the Hilbert function up to degree $t+3$.    
Note that the Hilbert function, in degree $t+3$, has value equal to
\[
1 + 2 + 3 + 4 + 5 + 6 \cdot [(t+3) - 4] = 6t+9.
\]
We now compute the value in degree $t+4 < 2t$.  Using the isomorphism (\ref{isom of ideals}), we have
\[
\begin{array}{rcl}
h_Z(t+4) & = & \binom{t+6}{2} - h^0({\mathcal I}_Z(t+4)) \\ \\
& = & \binom{t+6}{2} - h^0({\mathcal I}_X (t+1)) \\ \\
& = & \binom{t+6}{2} - \left [ \binom{t+3}{2} - h_X(t+1) \right ] \\ \\
& = & \binom{t+6}{2} - \left [ \binom{t+3}{2} - (3t+\delta) \right ] \\ \\
& = & 6t + 12 + \delta
\end{array}
\]
Then we easily see that $\Delta h_Z(t+4) = 3+\delta$ as claimed.

Next we compute the value in degree $t+5$.  We have 
$2t + \delta > t+5$, so we can use Claim 1.  Then a similar computation gives
\[
h_Z(t+5) = 6t+15+\delta.
\]
From this we immediately confirm  $\Delta h_Z(t+5) = 3$.  

Since $F$ is a common factor in all components $< 2t+\delta$, and 
since $\Delta h_Z$ takes the value 3 already in degree $t+5$, it 
repeats this value until $F$ is no longer a common factor.  In 
particular, it takes the value 3 up to degree $2t+\delta -1$.  

We now have to see what happens past degree $2t+\delta -1$.  
Note that using our above calculations, it follows that 
\[
\begin{array}{rcl}
h_Z(2t+\delta-1) & = & 6t+12+\delta + 3[2t+\delta-1 - (t+4)] \\ \\
& = & 3 ( 3t+ \delta) - 3 + \delta.
\end{array}
\]
Since $\deg Z = 3(3t+\delta)$, we have reached the multiplicity 
minus $(3-\delta)$.  We consider these cases separately.  
When $\delta = 1$ or $\delta = 2$, we are adding only 2 or 1, 
respectively, and since the first difference of the Hilbert function 
cannot be flat at this point, $\Delta h_Z$ must be as claimed in (a). This completes (a).
Since the sum of the values of $\Delta h_Z$ up to degree $2t-1$ is $9t-3$, this observation that $\Delta h_Z$  cannot be flat at this point also proves that the possibilities listed in (b) are the only ones possible.

 If $\delta = 0$, though, $\Delta h_Z$ can either end $\dots 3,3,0$ or 
 $\dots 3,2,1$.  We now consider these two possibilities.  The former 
 means that also in degree $2t+\delta = 2t$, all forms in $I_Z$ have 
 $F$ as a factor.  The latter means that there is a form, $G$, of degree 
 $2t+\delta = 2t$ in $I_Z$ that does not have $F$ as a factor, and 
 hence $(F, G)$ is a regular sequence (since $F$ is irreducible).

Suppose that the latter holds.  Note that the complete intersection 
defined by $(F,G)$ has degree $3 \cdot 2t = 6t = 2n$.  As in Claim 1, 
$G$ cuts out on $C$ a divisor of degree at least $2n$, so in fact $G$ 
cuts out exactly the divisor $2X$ on $C$.  So $X$ itself is not a 
complete intersection (since it has the Hilbert function given by 
(\ref{hf of 1 2 3 ... 3})), but the divisor $2X$ (as a subscheme of 
$\pr2$) is a complete intersection, namely of type $(3,2t)$.  
Note that $2X$, which is curvilinear, is not the same as $Z$.

Now suppose that $C$ is smooth.  We know that then two effective 
divisors of the same degree are linearly equivalent if and only if 
they have the same sum in the group of $C$.  The condition described 
in the previous paragraph implies that the sum of the points of $X$ is a 
2-torsion point in the group of $C$ but is not zero.  Since there are at most three 2-torsion points in the group of $C$,  
for general choices we have a contradiction, and so such a $G$ cannot exist (in general), 
and we have proved the assertion about the general choice of the points.

Finally, we show that the Hilbert function ii.\ of (b) also occurs.  We begin with four general lines, 
$\lambda_1, \lambda_2,\lambda_3,\lambda_4 \subset \pr2$ 
and let $P_1,P_2, P_3, P_4, P_5, P_6$ be the six points of pairwise 
intersection of these lines.  Let $G_1$ be the form defining the union 
of these four lines.  Let $X_1 = \bigcup_{1 \leq i \leq 6} P_i$.  Notice 
that $X_1$ does not lie on any conic, since by B\'ezout any conic 
containing $X_1$ has to contain all four lines $\lambda_1,\dots,\lambda_4$, 
hence must have $G_1$ as a factor.  Hence  Hilbert function of $X_1$ has first difference $(1,2,3)$, and$X_1$ is not a complete intersection.  

Let $C$ be a general cubic curve containing $X_1$, and let $F$ be the 
defining polynomial of $C$.  $C$ is smooth.  Notice that the degree of 
the complete intersection of $F$ and $G_1$ is 12, and this complete 
intersection is at least double at each $P_i$, so in fact it is exactly 
double at each $P_i$.  In particular, there is no additional multiplicity 
at any of the $P_i$ coming from tangency.  As a divisor on $C$, note 
that $X_1$ is not cut out by any conic, since it is not a complete intersection.  
However, the divisor $2X_1$ is cut out by a quartic, namely $G_1$.  

Now let $X$ be the union of $X_1$ with a general hypersurface section, 
$W_1$, of $C$ cut out by a curve of degree $t-2$.  Note that $W_1$ is a 
complete intersection defined by $(F,H)$ for some form $H$ of degree 
$t-2$.  We first claim that $X$ is not a complete intersection.  Indeed, 
suppose that $X$ were a complete intersection defined by $(F,H')$ for 
some $H'$ of degree $t$.  Then $I_X$ links $W_1$ to $X_1$.  But $W_1$ 
and $X$ are both complete intersections sharing a generator, so by 
liaison theory the residual is also a complete intersection.  But we 
have seen that $X_1$ is not a complete intersection.  Contradiction.  
In particular, $\Delta h_X$ is given by (\ref{hf of 1 2 3 ... 3}).

Now let $Z$ be the fat point scheme supported on $X$, and consider 
the form $G_1 H^2$.  This has degree $2t$, and cuts out the divisor 
$2X$ on $C$.  Even more, $G_1 H^2$ is an element of $I_Z$ in 
degree $2t$ that does not have $F$ as a factor.  As we saw above, 
this gives a value $\Delta h_Z(2t) = 2$ and $\Delta h_Z(2t+1) = 1$ 
as desired.  This completes the proof of Theorem~\ref{smooth cubic}.
\end{proof}

Now we wish to explore the possibilities when $C$ is singular and one 
point, $P$, of $X$ is the singular point of $C$.  The arguments are very 
similar, and we will primarily  highlight the differences.  The main 
observation is that $C$ is already double at $P$ so we have to focus 
on the remaining $n-1$ points.

\begin{lem} \label{ci poss}
Assume that $C$ is singular, that $P \in X \subset C$ is the singular 
point of $C$, and that $n \geq 5$.  Then $X$ is not a complete intersection.
\end{lem}

\begin{proof}
More precisely, we will show that if $P \in X \subset C$ with $X$ a 
complete intersection, and if $P$ is the singular point of $C$, then 
$X$ has one of the following types:  $CI(1,1),\ CI(1,2),\ CI(2,2)$.

First note that if $X$ is a complete intersection defined by forms $(F,G)$, 
where $F$ is the defining polynomial for $C$, then $X$ has multiplicity 
$\geq 2$ at $P$, so $X$ is not reduced.  Hence we have to determine 
all the possibilities for reduced complete intersections on $C$ that do 
not use $F$ as a minimal generator.  The listed possibilities are clear: 
one point, two points, four points, and these all exist even including 
$P$ as one of the points.  Using the irreducibility of $F$, it is not hard to 
show that these are the only possibilities, and we omit the details.
\end{proof}

\begin{thm} \label{singular cubic}
Assume that $C$ is an irreducible singular cubic with singular point $P$, 
and assume that $P \in X$, where $X$ is a reduced set of $3t + \delta$ 
points of $C$, with $0 \leq \delta \leq 2$.  Assume further that $t > 3$.  
Then the Hilbert function of the double point scheme $Z$ supported on $X$ is as follows.

\begin{enumerate}

\item If $\delta = 0$ then 
{\small 
\[
\begin{array}{l|ccccccccccccccccccccccccccccccccccc}
\deg & 0 & 1 & 2 & 3 & 4 & 5  & \dots & t+2 & t+3 & t+4 & t+5 & \dots & 2t & 2t+1 & 2t+2  \\
\hline
\Delta h_Z & 1 & 2 & 3 & 4 & 5 & 6 &  \dots & 6 & 5 & 3 & 3 & \dots & 3 & 1 & 0
\end{array}
\]
}

\item If $\delta = 1$ then either 
{\small 
\[
\begin{array}{l|ccccccccccccccccccccccccccccccccccc}
\deg & 0 & 1 & 2 & 3 & 4 & 5  & \dots & t+2 & t+3 & t+4 & t+5 & \dots & 2t & 2t+1 & 2t+2  \\
\hline
\Delta h_Z & 1 & 2 & 3 & 4 & 5 & 6 &  \dots & 6 & 6 & 3 & 3 & \dots & 3 & 3 & 0
\end{array}
\]
}
or
{\small 
\[
\begin{array}{l|ccccccccccccccccccccccccccccccccccc}
\deg & 0 & 1 & 2 & 3 & 4 & 5  & \dots & t+2 & t+3 & t+4 & t+5 & \dots & 2t & 2t+1 & 2t+2  \\
\hline
\Delta h_Z & 1 & 2 & 3 & 4 & 5 & 6 &  \dots & 6 & 5 & 4 & 3 & \dots & 3  & 3 & 0
\end{array}
\]
}

\item If $\delta = 2$ then
{\small 
\[
\begin{array}{l|ccccccccccccccccccccccccccccccccccc}
\deg & 0 & 1 & 2 & 3 & 4 & 5  & \dots & t+2 & t+3 & t+4 & t+5 & \dots & 2t & 2t+1 & 2t+2 & 2t+3  \\
\hline
\Delta h_Z & 1 & 2 & 3 & 4 & 5 & 6 &  \dots & 6 & 6 & 4 & 3 & \dots & 3 & 3 & 2 & 0
\end{array}
\]
}
\end{enumerate}

\end{thm}

\begin{proof}

The bound $t > 3$ is simply to ensure that in each case, some value of the Hilbert 
function $\Delta h_Z$ takes the value 3.  For instance, in the case $\delta = 0$, 
we have $2t > t+3$. As a consequence of Lemma \ref{ci poss}, when 
$n = 3t + \delta \geq 5$ the Hilbert function of $X$ must have first difference
\addtocounter{thm}{1}
\begin{equation} \label{hf of X on sing}
\begin{array}{l|ccccccccccccccccccccccccccccccccccc}
\deg & 0 & 1 & 2 & 3 & \dots & t & t+1 & t+2   \\
\hline
\Delta h_X & 1 & 2 & 3 & 3 & \dots & 3 & \delta & 0
\end{array}
\end{equation}
 
 In analogy with Theorem \ref{smooth cubic}, we first have
 
 \medskip
 
 \noindent {\bf Claim 1:} {\em Assume that
 \[
 \ell \leq 
 \left \{
 \begin{array}{ll}
 2t & \hbox{if $\delta = 0$} \\ \\
 2t+1 & \hbox{if $\delta = 1,2$}
 \end{array}
 \right.
 \]
Then $F$ is a common factor of $(I_Z)_\ell$.}
 
The proof is the same as that of Claim 1 in Theorem \ref{smooth cubic}, 
except that the intersection multiplicity of $F$ and $G$ at $P$ is now at least 4.
 
\noindent {\bf Claim 2:} {\em For $\ell \leq t+2$, $(I_Z)_\ell$ has $F^2$ as a common factor.  Furthermore,

\bigskip

\begin{itemize}

\item If $\delta = 0$ then $F^2$ is {\em not} a common factor of $(I_Z)_{t+3}$.

\item If $\delta = 2$ then $F^2$ is a common factor of $(I_Z)_{t+3}$.

\item If $\delta =1$ then $F^2$ may or may not be a common factor of $(I_Z)_{t+3}$ (examples exist for either option).

\end{itemize}
}

The proof of Claim 2 hinges on the possible Hilbert functions for $X-\{P\}$.  
In particular, we show that $(I_{X-\{P\}})_{t-1}$ always has $F$ as a common 
factor, and the differences in the three cases rest with the possibilities for $(I_{X-\{P\}})_t$, 
which we get by comparing to those for $I_X$, obtained using Lemma \ref{ci poss}.

\begin{itemize}

\item If $\delta = 0$ then $X$ has Hilbert function with first difference

\[
\begin{array}{l|ccccccccccccccccccccccccccccccccccc}
\deg & 0 & 1 & 2 & 3 & \dots & t & t+1  \\
\hline
\Delta h_X & 1 & 2 & 3 & 3 & \dots & 3  & 0
\end{array}
\]
so clearly the only possibility for $\Delta h_{X-\{P\}}$ is
\[
\begin{array}{l|ccccccccccccccccccccccccccccccccccc}
\deg & 0 & 1 & 2 & 3 & \dots & t-1 & t & t+1   \\
\hline
\Delta h_{X-\{P\}} & 1 & 2 & 3 & 3 & \dots & 3 & 2 & 0
\end{array}
\]
Hence there is a form $G$ of degree $t$ vanishing on $X-\{P\}$ but not 
containing $F$ as a factor, so $FG \in (I_Z)_{t+3}$ does not have $F^2$ as a factor.\\

\item If $\delta = 2$ then $X$ has Hilbert function with first difference

\[
\begin{array}{l|ccccccccccccccccccccccccccccccccccc}
\deg & 0 & 1 & 2 & 3 & \dots & t-1 & t & t+1 & t+2   \\
\hline
\Delta h_{X-\{P\}} & 1 & 2 & 3 & 3 & \dots & 3 & 3 & 2 & 0
\end{array}
\]

\noindent so $\Delta h_{X-\{P\}}$ is

\[
\begin{array}{l|ccccccccccccccccccccccccccccccccccc}
\deg & 0 & 1 & 2 & 3 & \dots & t-1 & t & t+1 & t+2   \\
\hline
\Delta h_{X-\{P\}} & 1 & 2 & 3 & 3 & \dots & 3 & 3 & 1 & 0
\end{array}
\]

\end{itemize}

\noindent We know that $(I_Z)_\ell \cong (I_{X - \{P\}})_{\ell -3}$ for $\ell$ satisfying 
the bounds of Claim 1, and  as a result of the above observations we know 
when $(I_{X - \{P\}})_{\ell -3}$ is forced to have $F$ as a common factor, so the claim follows.

\item If $\delta = 1$ then $X$ has Hilbert function with first difference 

\[
\begin{array}{l|ccccccccccccccccccccccccccccccccccc}
\deg & 0 & 1 & 2 & 3 & \dots & t & t+1 & t+2 \\
\hline
\Delta h_X & 1 & 2 & 3 & 3 & \dots & 3  & 1 & 0
\end{array}
\]

\noindent so $\Delta h_{X-\{P\}}$ is either

\[
\begin{array}{l|ccccccccccccccccccccccccccccccccccc}
\deg & 0 & 1 & 2 & 3 & \dots & t-1 & t & t+1   \\
\hline
\Delta h_{X-\{P\}} & 1 & 2 & 3 & 3 & \dots & 3 & 3 & 0
\end{array}
\]

or

\[ 
\begin{array}{l|ccccccccccccccccccccccccccccccccccc}
\deg & 0 & 1 & 2 & 3 & \dots & t-1 & t & t+1 & t+2  \\
\hline
\Delta h_{X-\{P\}} & 1 & 2 & 3 & 3 & \dots & 3 & 2 & 1 & 0
\end{array}
\]

\noindent Since we have removed $P$, the remaining points could 
be a complete intersection, so $F^2$ is a common factor of $(I_Z)_{t+3}$ if and only 
if the points of $X - \{P\}$ are not a complete intersection of a curve of degree $t$ with $F$. 
This completes the proof of Claim 2.

The rest of the proof is very similar to that of Theorem \ref{smooth cubic} and we omit the details.
\end{proof}


\section{Approach II: Points on Cubics}\label{appII}

Let $Z = m_1p_1 + \cdots + m_n p_n \subset \mathbb P^2$, 
where the points $p_1,\dots,p_n$ are distinct and arbitrary.
When $n<9$, a complete determination of $h_Z$ 
is given in \cite{GHM}, but the case of $n\geq9$ remains of interest.
Giving a complete determination of $h_Z$
for all $n\geq9$ arbitrary distinct points $p_1, \ldots , p_n$
would involve solving some extremely hard
open problems.
For example, it is even an open problem to determine $h_Z$ for $n>9$
when the points $p_1, \ldots , p_n$ are general and $m_1=\cdots =m_n$.
So here, as in \textsection\ref{ptsoncubics},
we consider the case of $n\geq9$ points $p_i$ in special cases. These cases include those
considered in \textsection\ref{ptsoncubics}. We recover and in some cases extend the
results of \textsection\ref{ptsoncubics}, but the methods we use here are different. 
To start, let $p_1,\cdots,p_n$ be $n\geq 9$ distinct points on a reduced plane
cubic $C$.  If $C$ is not irreducible, we assume further that all the points 
are smooth points of $C$.  If $D$ is  a component of $C$, let $n_D$ be the number of
these points on $D$. We will say that the points are {\em evenly
distributed\/} if $n_D=n(\deg(D))/3)$ for every reduced 
irreducible component $D$ of $C$.  Note that for $n$ points to be 
evenly distributed, it is necessary either that 3 divide $n$ or that $C$ 
be irreducible.

We will use some facts about surfaces obtained by blowing up points in the plane,
in particular we'll make use of the intersection form on such surfaces, which we now briefly recall.
Given distinct points $p_1,\ldots,p_n\in\pr2$, let $\pi : Y\to {\bf P}^2$ 
be the morphism obtained by blowing up the points $p_i$. The divisor class group ${\rm Cl}(Y)$
of divisors modulo linear equivalence
is a free abelian group with basis $[L],[E_1],\ldots,[E_n]$, where
$L$ is the pullback to $Y$ of a general line, and $E_i=\pi^{-1}(p_i)$.
There is a bilinear form, called the {\em intersection form}, defined on the group of divisors,
which descends to ${\rm Cl}(Y)$. It is uniquely determined by the fact
that $L$, $E_1$, $\ldots$, $E_n$ are orthogonal with respect to the intersection form,
with $L\cdot L=L^2=1$ and $E_i^2=-1$ for $i=1,\ldots,n$. For two distinct, reduced, irreducible curves 
$C_1$ and $C_2$ on $Y$, $C_1\cdot C_2$ is just the number of points of intersection
of the two curves, counted with multiplicity.
We recall that a divisor $F$ is {\em nef} 
if $F\cdot C\geq 0$ for every effective divisor $C$. A useful 
criterion for nefness is that if $F$ is an effective divisor such 
that $F\cdot C\geq0$ for every component $C$ of $F$, then $F$ is 
nef.

In preparation for stating Theorem \ref{method2cor1}, our main result in this section, we
set some additional notation. Let $Z=m(p_1+\cdots+p_n)$. In 
degrees $t$ such that $3t=mn$,
the value of $h_Z(t)$ is influenced by torsion in the group $\Pic(C)$.
Our formula for $h_Z$ as given in Theorem \ref{method2cor1}
accounts for this influence via an integer-valued function we will denote by $s$.
In fact, $s$ depends on the points $p_i$, on $m$ and on $t$, but for a fixed
set of points $p_i$ it is convenient to mostly suppress the dependence on the points
and denote $s$ as $s(t,n,m)$, where the parameter $n$ is a reminder of the dependence on the $n$ points. To define $s(t,n,m)$,
let $L$ be a general line in the plane and fix evenly distributed smooth points
$p_1,\ldots,p_n$ of a reduced cubic $C$. Since Theorem \ref{method2cor1} applies only for
$n\geq9$ and we need $s(t,n,m)$ only when $t\geq nm/3$, we define $s(t,n,m)$ only for $n\geq9$
when $t\geq nm/3$:

\begin{itemize}
\item[(1)] If $t>nm/3$, we set $s(t,n,m)=0$.
\item[(2)] If $n=9$ and $t=3m$, let $\lambda$ be the order (possibly infinite) of
${\mathcal O}_C(3L)\otimes {\mathcal O}_C(-p_1-\cdots-p_9)$ in $\Pic(C)$.
We then set $s(t,n,m)=\lfloor m/\lambda\rfloor$. 
\item[(3)] If $n>9$ and $t=nm/3$, we set $s(t,n,m)=1$ if 
${\mathcal O}_C(tL)\otimes {\mathcal O}_C(-mp_1-\cdots-mp_n)={\mathcal O}_C$ 
in $\Pic(C)$, and we set $s(t,n,m)=0$ otherwise. 
\end{itemize}

The value of $s(t,n,m)$ depends on whether
${\mathcal O}_C(tL)\otimes {\mathcal O}_C(-mp_1-\cdots-mp_n)$ is trivial.
Note triviality of this line bundle is equivalent to the divisor $mp_1+\cdots+mp_n$ on $C$ 
being the intersection of $C$ with a curve $H$, 
necessarily of degree $t=mn/3$. Of course it can happen that 
${\mathcal O}_C(tL)\otimes {\mathcal O}_C(-p_1-\cdots-p_n)$ is non-trivial even though
${\mathcal O}_C(tmL)\otimes {\mathcal O}_C(-mp_1-\cdots-mp_n)$ is trivial. For example,
if $p_1, p_2$ and $p_3$ are flexes on $C$ but not collinear, then 
${\mathcal O}_C(L)\otimes {\mathcal O}_C(-p_1-p_2-p_3)$ is not trivial, but
${\mathcal O}_C(3L)\otimes {\mathcal O}_C(-3p_1-3p_2-3p_3)$ is trivial, and $H$ in this case
is the union of the lines tangent to $C$ at the points $p_1$, $p_2$ and $p_3$. 
When $C$ is a  smooth cubic curve, triviality of 
${\mathcal O}_C(tL)\otimes {\mathcal O}_C(-mp_1-\cdots-mp_n)$
is equivalent to the sum $mp_1+\cdots+mp_n$ being trivial in the group law on the cubic
(with respect to a flex being taken as the identity element).   (The divisor $X_1$ given in the proof of part (c) of Theorem \ref{smooth cubic} gives another example, and shows that this issue arose also with the first approach.)

\begin{rem}
When $n=9$, the values of $\lambda$ that can occur depend on
the torsion in $\Pic(C)$, and this depends on $C$ and on the characteristic
of the ground field; see Remark \ref{torsionRem}.
Thus knowing something about $C$ tells us something about
what Hilbert functions can occur for points on $C$, but the Hilbert 
functions themselves depend only on $\lambda$, and
already for a smooth irreducible non-supersingular cubic $C$,
there is torsion of all orders. 
\end{rem}

\begin{thm}\label{method2cor1}
Let $X =p_1+\cdots+p_n$  be a set of $n\geq 9$
evenly distributed smooth points on a reduced plane cubic $C$.  Let $Z = mX$.  
The value $h_Z(t) = \hbox{dim} (k[{\bf P}^2]/(I(Z)))_t$
of the Hilbert function in degree $t$ is:
\begin{itemize}
\item[(i)] $\binom{t+2}{2}$ if $t<3m$;
\item[(ii)] $n\binom{m+1}{2}-s(t,n,m)$ if $t \geq nm/3$; and
\item[(iii)] $\binom{t+2}{2}-\binom{t-3r+2}{2}+n\binom{m-r+1}{2}-s(t-3r,n,m-r)$ if $n>9$ and
$3m\leq t< mn/3$, where $r=\lceil (mn-3t)/(n-9)\rceil$.
\end{itemize}
\end{thm}

\begin{proof} This result is a corollary of the main result of
\cite{anticanSurf}. 
Let $F=tL-mE_1-\cdots-mE_n$, where $\pi : Y\to {\bf P}^2$ 
is the morphism obtained by blowing up the points $p_i$,
$L$ is the pullback to $Y$ of a general line, and $E_i=\pi^{-1}(p_i)$.

Let $C'\subset Y$ be the proper transform of $C$ with respect to $\pi$. 
Since the points $p_i$ blown up are smooth points on $C$, we see
$[C']=[3L-E_1-\cdots-E_n]$ (and hence $C'$ is an anticanonical divisor). Moreover,
each component of $C'$ is the proper transform
$D'$ of a component $D$ of $C$, and each of the components of $C$ (and hence of $C'$) is 
reduced. (To see this note that $n_D>0$ for each component $D$ of $C$ 
since the points $p_i$ are evenly distributed, but the number of points $p_i$ which lie on $D$
is $n_D$ and all of the points $p_i$ are smooth points of $C$, so each component of $C$ has a smooth point
and hence must be reduced.)

In addition, the following statements are equivalent:
\begin{itemize}
\item[(a)] $F\cdot D'\geq 0$ for every irreducible component $D'$ of $C'$; 
\item[(b)] $F\cdot C'\geq 0$; and
\item[(c)] $F\cdot D'\geq 0$ for some irreducible component $D'$ of $C'$.
\end{itemize}
Clearly, (a) implies (b), and (b) implies (c). We now show that (c) implies (a).
If $C'$ has only one component, then (c) and (a) are 
trivially equivalent, so suppose $D_1'$ and $D_2'$ are distinct components of $C'$.
In order to show that $F\cdot D_1'\geq0$ implies $F\cdot D_2'\geq0$,
we will use the assumption that the points $p_i$ are evenly distributed smooth points of $C$.
Let $D_j=\pi(D_j')$, so $D_j'$ is the proper transform of $D_j$.
Because the points are evenly distributed, we have
$n_{D_j}=n(\deg(D_j))/3$. Thus $n_{D_j}$ of the $n$ points $p_i$
lie on $D_j$. Because the points are smooth points of $C$, we have
$[D_j']=[\deg(D_j)L-\sum_{p_i\in D_j}E_i]$, where the sum involves $n_{D_j}$ terms. 
Thus $F\cdot D_j'\geq 0$ can be rewritten as $t\deg(D_j)-mn_{D_j}\geq0$.
Substituting $n(\deg(D_j))/3$ for $n_{D_j}$ gives $t\deg(D_j)-mn(\deg(D_j))/3\geq0$
which is equivalent to $3t-mn\geq0$, which is itself just $F\cdot C'\geq0$.
Thus $F\cdot D_1'\geq0$ and $F\cdot D_2'\geq0$ are both equivalent to $F\cdot C'\geq0$,
and hence $F\cdot D_1'\geq0$ if and only if $F\cdot D_2'\geq0$. This shows (c) implies (a).

We now show that $h^0(Y,{\mathcal O}_Y(F))=0$ 
if and only if $t<3m$. For $t\geq 3m$, we have ${\mathcal O}_Y(F)
={\mathcal O}_Y((t-3m)L+mC')$, and hence $h^0(Y,{\mathcal O}_Y(F))>0$.
If, however, $t<3m$, then $3t<9m\leq nm$ so
$F\cdot C'<0$, and hence, as we saw above,
$F\cdot D'< 0$ for each component $D'$ of $C'$, in which case  
each component $D'$ of $C'$ is a fixed component of $|F|$ so
$h^0(Y, {\mathcal O}_Y(F))=h^0(Y, {\mathcal O}_Y(F-C'))
=h^0(Y, {\mathcal O}_Y((t-3)L-(m-1)E_1-\cdots-(m-1)E_n))$.
But $t-3<3(m-1)$, so, by the same argument, we can again
subtract off $C'$ without changing $h^0$. Continuing
in this way we eventually obtain $h^0(Y, {\mathcal O}_Y(F))=
h^0(Y, {\mathcal O}_Y((t-3m)L))=h^0(\pr2, {\mathcal O}_{\pr2}(t-3m))$, but 
$h^0(\pr2, {\mathcal O}_{\pr2}(t-3m))=0$ since $t-3m<0$. Thus 
$h_{Z}(t)=\binom{t+2}{2}$ for $t<3m$, which proves (i).

Next consider (ii). If $t\geq nm/3$, i.e., if $F\cdot C'\geq0$, then as we saw above
$F\cdot D'\geq0$ for every component $D'$ of $C'$.
But as we also saw above, $(t-3m)L+mC'\in |F|$, hence $F$ is nef.
If $t>nm/3$ (in which case $s(t,n,m)=0$), then $F\cdot C'>0$, so
by \cite[Theorem III.1(a,b)]{anticanSurf}, $h^1(Y,{\mathcal O}_Y(F))=0$.
Thus \eqref{RRoch2} gives $h_{Z}(t)=n\binom{m+1}{2}=n\binom{m+1}{2}-s(t,n,m)$ as claimed.
We are left with the case that $t=nm/3$. 

Suppose $t=nm/3$ and $n=9$. Thus $F=mC'$ and
$F\cdot C'=0$ (because $n=9$ and $t=3m$), so $(C')^2=0$.
By duality we have $h^2(Y, {\mathcal O}_Y(mC'))=h^0(Y, {\mathcal O}_Y(-(m+1)C'))=0$,
so by Riemann-Roch we have 
$h^0(Y, {\mathcal O}_Y(F))-h^1(Y, {\mathcal O}_Y(F))=1+(F^2+C'\cdot F)/2=1$.
Since $F$ is nef, so is $iC'$ for all $i\geq0$. Since $F\cdot C'=0$, either
$|F|$ has an element disjoint from $C'$ or $F$ and $C'$ share a common component.

If $|F|$ has an element disjoint from $C'$, then ${\mathcal O}_{C'}(F)$ is trivial, so
$h^0(C', {\mathcal O}_{C'}(F))=1$ since $C$ (and hence $C'$) is connected and 
reduced.
Suppose $F$ and $C'$ share a common component. Then $C'$ is in the base 
locus of $|F|$ by \cite[Corollary III.2]{anticanSurf}, and hence 
$h^0(Y, {\mathcal O}_Y(F))=h^0(Y, {\mathcal O}_Y(F-C'))$.
Let $\phi$ be the least $i>0$ (possibly infinite) such that
$C'$ is not in the base locus of $|iC'|$. Then 
we have $h^0(Y, {\mathcal O}_Y(jC'))=h^0(Y, {\mathcal O}_Y((j+1)C'))$
for $0\leq j<\phi-1$, so by induction (using the base case
$h^0(Y, {\mathcal O}_Y)=1$ and the fact $h^0(Y, {\mathcal O}_Y(jC'))-h^1(Y, {\mathcal O}_Y(jC'))=1$)
we have $h^0(Y, {\mathcal O}_Y(jC'))=1$ and $h^1(Y, {\mathcal O}_Y(jC'))=0$ for all $0\leq j<\phi$. 
It follows that
$$0\to {\mathcal O}_Y((s-1)C')\to {\mathcal O}_Y(sC')\to {\mathcal O}_{C'}(sC')\to 0 \eqno(\star)$$
is exact on global sections for $1\leq s\leq \phi$, and that $h^0(C', {\mathcal O}_{C'}(sC'))=0$
and $h^1(C', {\mathcal O}_{C'}(sC'))=0$ for $0<s<\phi$. Thus ${\mathcal O}_{C'}(sC')$
is nontrivial for $0<s<\phi$. Since for all $m$, $|mC'|$ either has an element disjoint from $C'$ or
$C'$ is in the base locus of $|mC'|$, we see that ${\mathcal O}_{C'}(\phi C')$ is trivial,
and hence $\phi$ is the order of ${\mathcal O}_{C'}(C')$ in $\Pic(C')$. But since the points
$p_i$ blown up are smooth points of $C$, the morphism $\pi:Y\to\pr2$ induces an isomorphism 
$C\to C'$, and under this isomorphism,
${\mathcal O}_C(3L)\otimes {\mathcal O}_C(-p_1-\cdots-p_n)$ 
corresponds to ${\mathcal O}_{C'}(C')$, so we see $\phi=\lambda$. 
It follows that  $h^0(C', {\mathcal O}_{C'}(sC'))=h^1(C', {\mathcal O}_{C'}(sC'))$
for all $s\geq0$, and these are both 1 if ${\mathcal O}_{C'}(sC')$ is trivial
(i.e., if $s$ is a multiple of $\lambda$) and they are 0 otherwise.

We now claim that $(\star)$ is exact on globals sections for all $s\geq 1$.
It is enough to show this
when $s$ is a multiple of $\lambda$, because otherwise, as we noted above,
$h^0(C', {\mathcal O}_{C'}(sC'))=0$ and hence
$(\star)$ is automatically exact on global sections.
But $|\lambda C'|$ (and hence also $|i\lambda C'|$ for all $i\geq1$) has an element disjoint from $C'$, so
$H^0(Y, {\mathcal O}_Y(i\lambda C'))\to H^0(C', {\mathcal O}_{C'}(i\lambda C'))$ is onto,
which shows that $(\star)$ is exact on global sections when $s$ is a multiple of $\lambda$.

It follows that $(\star)$
is also exact on $h^1$'s (since as above $h^2(Y, {\mathcal O}_Y(iC'))=0$
for all $i\geq0$), and hence that
$h^1(Y, {\mathcal O}_Y(mC'))=h^1(Y, {\mathcal O}_Y)+
\sum_{1\leq i\leq m}h^1(C', {\mathcal O}_{C'}(iC'))$. 
Now $h^1(Y, {\mathcal O}_Y)=h^1(\pr2, {\mathcal O}_{\pr2})=0$
and $h^1(C', {\mathcal O}_{C'}(iC'))$ is 1 if and only if $i$ is a multiple of $\lambda$
and it is 0 otherwise. Thus $h^1(Y, {\mathcal O}_Y(mC'))$ is the number of summands
$h^1(C', {\mathcal O}_{C'}(iC'))$ for which $i$ is a multiple of $\lambda$; i.e.,
$h^1(Y, {\mathcal O}_Y(mC'))=\lfloor m/\lambda\rfloor$, which is just $s(t,n,m)$.
This implies that $h_{Z}(t)=n\binom{m+1}{2}-s(t,n,m)$, as claimed.

If $t=nm/3$ but $n>9$, then $F^2>0$ so by \cite[Theorem III.1(c)]{anticanSurf}
either ${\mathcal O}_{C'}(F)$ is trivial (in which case $s(t,n,m)=1$)
and $h^1(Y,{\mathcal O}_Y(F))=1$ (and hence 
$h_{Z}(t)=n\binom{m+1}{2}-1=n\binom{m+1}{2}-s(t,n,m)$),
or $C'$ is in the base locus of $|F|$. If $C'$ is in the base locus,
then by \cite[Theorem III.1(d)]{anticanSurf} and the fact that $F^2>0$ we have
${\mathcal O}_{C'}(F)$ is not trivial (in which case $s(t,n,m)=0$)
and $h^1(Y,{\mathcal O}_Y(F))=0$, and hence 
$h_{Z}(t)=n\binom{m+1}{2}-s(t,n,m)$, as claimed. 

Now consider case (iii); i.e., $3m\leq t<nm/3$ and $n>9$. Then
$F\cdot D'< 0$ for each component $D'$ of $C'$ (since
the points are evenly distributed), in which case  
$h^0(Y, {\mathcal O}_Y(F))=h^0(Y, {\mathcal O}_Y(F-C'))
=h^0(Y, {\mathcal O}_Y((t-3)L-(m-1)E_1-\cdots-(m-1)E_n))$.
If $t-3 < n(m-1)/3$, we can subtract $C'$ off again. 
This continues  until we have subtracted $C'$ off $r=\lceil (mn-3t)/(n-9)\rceil$ times,
at which point we have that $F-rC'$ is nef and effective
and $h^0(Y, {\mathcal O}_Y(F))=h^0(Y, {\mathcal O}_Y(F-rC'))$.
Applying (ii) to $F-rC'$ gives
$\binom{t-3r+2}{2}-h^0(Y, {\mathcal O}_Y(F-rC'))=h_{(m-r)Z}(t-3r)=n\binom{m-r+1}{2}-s(t-3r,n,m-r)$
or $h^0(Y, {\mathcal O}_Y(F-rC'))=\binom{t-3r+2}{2}-(n\binom{m-r+1}{2}-s(t-3r,n,m-r))$.
Substituting this in for $h^0(Y, {\mathcal O}_Y(F))$ in 
$h_{Z}(t)=\binom{t+2}{2}-h^0(Y, {\mathcal O}_Y(F))$ gives (iii).
\end{proof}

\begin{rem}
We can now write down all possible Hilbert functions for
$n\geq 9$ points of multiplicity $m$ for each possible choice of Hilbert function
for the reduced scheme given by the points, if the points are smooth points of a reduced cubic curve
and evenly distributed. Suppose $X=p_1+\cdots+p_n$ and $m=1$.
If $3$ does not divide $n$, or it does but 
$s(n/3,n,1)=0$, then the difference function
for the Hilbert function of $X$ is the same as given in \eqref{hf of 1 2 3 ... 3}, but
if $3$ divides $n$ and $s(n/3,n,1)=1$, then $X$ is a complete intersection 
and the difference function
for the Hilbert function of $X$ is the same as given in \eqref{hf  of 1 2 3 ... 3 2 1}.

We now compare our results for $Z = 2X=2(p_1+\cdots+p_n)$ with those obtained in Proposition \ref{ci case} and Theorem \ref{smooth cubic}, and we explicitly list those cases skipped there (because there we assumed $n=3t$ with $t\geq3$ in Proposition \ref{ci case} and $n=2t+\delta$ with
$t>5-\delta$ in Theorem \ref{smooth cubic}).

Say $n\equiv 1\mod 3$. Then
the difference function for the Hilbert function is:

$n=10$: 1 2 3 4 5 6 6 3 0

$n=13$: 1 2 3 4 5 6 6 6 4 2 0, and for

$n=10+3x$ for $x>1$: the result is the same as given in Theorem \ref{smooth cubic}(a).

Next, say $n\equiv 2\mod 3$. Then the difference function for the Hilbert function is:

$n=11$: 1 2 3 4 5 6 6 5 1 0, and for

$n=11+3x$ for $x>0$: the result is the same as given in Theorem \ref{smooth cubic}(a).

If $n=3x$, there are two possibilities.
If $s(2x,n,2)=0$ for the given points (i.e.,
the divisor $2p_1+\cdots+2p_n$ on $C$ is not cut out by a curve of degree 
$2x$, or equivalently
${\mathcal O}_C(2xL-2E_1-\cdots-2E_n)$ is not trivial), 
then the difference function for the Hilbert function is:

$n=9$: 1 2 3 4 5 6 6 0, and

$n=3x$: 1 2 3 4 5 6 $\ldots$ 6 3 $\ldots$ 3 0 
for $x\geq4$,
where the number of 6's is $x-1$
and the number of trailing 3's is $x-3$.
For $x>5$, this is the same as the result given in Theorem \ref{smooth cubic}(b). 

If $s(2x,n,2)=1$ for the given points (i.e.,
the divisor $2p_1+\cdots+2p_n$ on $C$ is cut out by a curve of degree 
$2x$, or equivalently ${\mathcal O}_C(2xL-2E_1-\cdots-2E_n)$ is trivial), 
but $s(n/3,n,1)=0$ 
(so $p_1+\cdots+p_n$ is not cut out by a curve of degree $x$, which is
equivalent to saying that ${\mathcal O}_C(xL-E_1-\cdots-E_n)$ is not trivial), 
then the difference function for the Hilbert function is:

$n=9$: 1 2 3 4 5 6 5 1 0,

$n=12$: 1 2 3 4 5 6 6 6 2 1 0, and

$n=3x$: 1 2 3 4 5 6 $\ldots$ 6 3 $\ldots$ 3 2 1 0
for $x>4$,
where the number of 6's is $x-1$
and the number of trailing 3's is $x-4$.
For $x>5$, this is the same as the result given in Theorem \ref{smooth cubic}(c). 

Now say $n=3x$ and $s(x,n,1)=1$. In this case, $X$ is the complete intersection of 
$C$ and a form of degree $t$, and the difference function for the Hilbert function of $2X$ is:

$n=9$: 1 2 3 4 5 6 4 2 0

$n=12$: 1 2 3 4 5 6 6 5 3 1 0 and

$n=3x$ for $x>4$: the result is the same as given in Proposition \ref{ci case}.
\end{rem}

\begin{rem}\label{torsionRem}
The possible values of the Hilbert functions as given in Theorem \ref
{method2cor1} depend partly on what torsion occurs in $\Pic(C)$, and
this in turn is affected by the characteristic of $k$. 
When $C$ is smooth,
see \cite[Example IV.4.8.1]{refHr} for a discussion of the torsion.
When $C$ is reduced but not smooth, the torsion is easy to understand since
it is all contained in the identity component $\Pic^0(C)$ of $\Pic(C)$,
whose group structure is isomorphic either 
to the additive or multiplicative groups of the ground field.
(See for example \cite[Proposition 5.2]{refHL}, which states a result for
curves of so-called canonical type. But for any reduced cubic $C$, one can always find a set of 9 evenly distributed 
smooth points of $C$, 
and the proper transform $C'$ with respect to blowing those points up
is a {\em curve of canonical type}, meaning that $C'\cdot D=K_X\cdot D=0$
for every component $D$ of $C'$. Since the points blown up are smooth  on $C$,
$C$ and $C'$ are isomorphic and thus so are $\Pic(C)$ and $\Pic(C')$, hence the conclusion of
\cite[Proposition 5.2]{refHL} applies to $C$, even though $C$ is not itself of canonical type.)
When $C$ is reduced and
irreducible but singular, for example, the result is that
$\Pic^0(C)$ is the additive
group of the ground field when $C$ is cuspidal and it is the
multiplicative
group of the field when $C$ is nodal \cite[Exercise II.6.9]{refHr}.
In particular, if $C$ is an irreducible cuspidal cubic curve over
a field of characteristic zero, then $\Pic^0(C)$ is torsion free, so
$h_
{2X}$ cannot be $(1,2,3,4,5,6,6,6,2,1)$; indeed, this follows, after
a simple calculation, because if ${\mathcal O}_{C}(2xL-2p_1-
\cdots-2p_n)$ is trivial, then so is ${\mathcal O}_{C}(xL-p_1-\cdots-
p_n)$.  On the other hand, ${\mathcal O}_{C}(xL-p_1-\cdots-p_n)$ can
be nontrivial even if ${\mathcal O}_{C}(2xL-2p_1-\cdots-2p_n)$ is
trivial
if the characteristic is 2 or if the singular point is a node but the
characteristic is not 2, since in those cases $\Pic(C)$ has elements of
order 2.
\end{rem}

\begin{rem} \label{method2cor2}
We can also use the method of proof of Theorem \ref{method2cor1}
to recover the result of Theorem \ref{singular cubic}
for the Hilbert function of $mX=m(p_1+\cdots+p_n)$ for $n\geq9$ points on a
reduced, irreducible cubic curve $C$ where $p_1$, say, is the singular
point and $m$ is 1 or 2. As is now clear, the approach of Theorem \ref{method2cor1}
is to determine $h^0(Y, {\mathcal O}_Y(tL-mE_1-\cdots-mE_n))$ for all $t$,
and then translate this into the Hilbert function or the difference function
for $mX$.

This translation is purely mechanical and the resulting Hilbert functions 
in the case that $n\geq 12$ are 
already given in Theorem \ref{singular cubic} (we leave writing down the Hilbert functions
for $9\leq n\leq 11$ using the results that follow as an exercise for the reader).
Thus it is the calculation of 
$h^0(Y, {\mathcal O}_Y(tL-mE_1-\cdots-mE_n))$
that is of most interest, and it is on this that we now focus.

Let $Y$ be the blow up of the points, let $C'$ be the proper
transform of $C$, and let $F_t=tL-E_1-\cdots-E_n$ and $G_t=tL-2(E_1+\cdots+E_n)$,
where we denote by $L$ both a general line in the plane and its pullback to $Y$.
Up to linear equivalence, note that $C'=3L-2E_1-E_2-\cdots-E_n$.

The goal here is to compute the values of $h^0(Y, {\mathcal O}_{Y}(F_t))$ and
$h^0(Y, {\mathcal O}_Y(G_t))$. For $t<3$, B\'ezout tells us that $h^0(X,{\mathcal O}_Y(F_t))=0$,
since $F_t\cdot C'<0$ (hence $h^0(Y, {\mathcal O}_Y(F_t))=h^0(Y, {\mathcal O}_Y(F_t-C'))$
and $(F_t-C')\cdot L<0$ (hence $h^0(Y, {\mathcal O}_Y(F_t-C'))=0$). If $t\geq3$, then certainly 
$h^0(Y,{\mathcal O}_Y(F_t))>0$, since $F_t=(t-3)L+C'+E_1$. We consider three cases, according to whether
$F_t\cdot C'<0$, $F_t\cdot C'>0$ or $F_t\cdot C'=0$.

If $0<F_t\cdot C'=3t-2-(n-1)$ (i.e., if $3\leq t<(n+1)/3$), 
then $h^0(Y,{\mathcal O}_Y(F_t))=h^0(Y,{\mathcal O}_Y(F_t-C'))=
h^0(Y,{\mathcal O}_Y((t-3)L+E_1))=h^0(Y,{\mathcal O}_Y((t-3)L))=
h^0(\pr2,{\mathcal O}_{\pr2}((t-3)L))=\binom{t-3+2}{2}$,
since $F_t-C'=(t-3)L+E_1$. If $F_t\cdot C'>0$ (i.e., $t>(n+1)/3$), then
$h^0(Y,{\mathcal O}_Y(F_t))=\binom{t+2}{2}-n$ (since $F_t$, meeting both
components of $-K_Y=C'+E_1$ positively, is nef and hence
$h^1(Y,{\mathcal O}_Y(F_t))=0$ by \cite[Theorem III.1(a,b)]{anticanSurf}).

This leaves the case that  $t = (n+1)/3$.  This means that $
{\mathcal O}_{C'}(F_t)$ has degree 0.  Consider the exact sequence
$$
0\to {\mathcal O}_Y((t-3)L+E_1)\to {\mathcal O}_Y(F_t)\to {\mathcal
O}_{C'}(F_t)\to 0.
$$
By an analogous argument to the one used to show $h^1(Y,{\mathcal O}
_Y(F_t))=0$ when $t > (n+1)/3$, we obtain that $h^1({\mathcal O}_Y
((t-3)L+E_1)) = 0$.  But $C'$ is a smooth rational curve, so also
the third sheaf in the sequence has vanishing first cohomology.
Thus we obtain $h^1(Y, O_Y(F_t))=0$, hence  the points impose
independent conditions.  It follows that  $h^0(Y, O_Y(F_t))=\binom
{t +2}{2}-n$ also for $t = (n+1)/3$.

We thus have: $h^0(Y,{\mathcal O}_Y(F_t))=0$ for $0\leq t <3$; 
$h^0(Y,{\mathcal O}_Y(F_t))=\binom{t-1}{2}$ for $3\leq t <(n+1)/3$;
and $h^0(Y,{\mathcal O}_Y(F_t))=\binom{t+2}{2}-n$ for $t\geq (n+1)/3$.

A similar analysis works for $2X$. There are now four ranges of degrees.
The first range is $t<6$, in which case $h^0(Y,{\mathcal O}_Y(G_t))=0$ by B\'ezout, arguing as above. 
For $t\geq6$, we have $h^0(Y,{\mathcal O}_Y(G_t))>0$, since up to linear equivalence
we have $G_t=(t-6)L+2(C'+E_1)$.
The second range is now $6\leq t< (n+8)/3$; in this case $2C'$ is, by B\'ezout,
a fixed component of $|G_t|$, so $h^0(Y,{\mathcal O}_Y(G_t))=
h^0(Y,{\mathcal O}_Y((t-6)L+2E_1))=\binom{t-4}{2}$.
The third range is $(n+8)/3\leq t <(2/3)(n+1)$, for which $C'$ is a fixed component of $|G_t|$
(and $G_t-C'=(t-6)L+C'+2E_1$ is nef)
so $h^0(Y,{\mathcal O}_Y(G_t))=h^0(Y,{\mathcal O}_Y(G_t-C'))$
and we know $h^0(Y,{\mathcal O}_Y(G_t-C'))$ by Theorem \ref{method2cor1}(ii) if $n>9$,
while $h^1(Y,{\mathcal O}_Y(G_t-C'))=0$ by \cite[Theorem III.1(a,b)]{anticanSurf}) if $n=9$,
so again we know $h^0(Y,{\mathcal O}_Y(G_t-C'))$.
The last range is $t\geq (2/3)(n+1)$, in which case $G_t$ is nef. 
If $t>(2/3)(n+1)$, then $G_t$ meets $-K_Y$ positively,
so $h^1(Y,{\mathcal O}_Y(G_t))=0$ \cite[Theorem III.1(a,b)]{anticanSurf}), 
and $h^0(Y,{\mathcal O}_Y(G_t))=\binom{t+2}{2}-3n$.
We are left with the case that $t=(2/3)(n+1)$. Consider the exact sequence
$$0\to {\mathcal O}_Y((t-3)L-E_2-\cdots-E_n)\to {\mathcal O}_Y(G_t)\to {\mathcal O}_{C'}(G_t)\to 0.$$
Since $G_t\cdot C'\geq0$ and $C'$ is smooth and rational, we have
$h^1(C',{\mathcal O}_{C'}(G_t))=0$, and since ${\mathcal O}_Y((t-3)L-E_2-\cdots-E_n)=
{\mathcal O}_Y((t-6)L+C'+2E_1)$ and $(t-6)L+C'+2E_1$ is nef (as observed above)
with  $(G_t-C')\cdot C'>0$, we have $h^1(Y,{\mathcal O}_Y(G_t-C'))=0$ \cite[Theorem III.1(a,b)]{anticanSurf})
and hence $h^1(Y,{\mathcal O}_Y(G_t))=0$, so in fact $h^0(Y,{\mathcal O}_Y(G_t))=\binom{t+2}{2}-3n$.
\end{rem}

\begin{rem}
Here we comment on what is left to do if one wants to recover the results of
section 2 using the methods of section 4.
So consider $n=9$ points on a given cubic $C$ (but note that there may be more  
than one cubic through the points), either all of multiplicity 1 or all of multiplicity 2.
The case that the points are evenly distributed smooth points of $C$
is done above, as is the case that the curve $C$ is reduced and irreducible. The case that the
points all lie on a conic follows from the known result for configuration types
of points on a conic \cite{GHM}. What's left is that the points do not all lie on
any given conic (and hence $C$ is reduced) and either: one or more of the points
is not a smooth point of $C$ and $C$ is not irreducible, or the points are
not distributed evenly (and hence again $C$ is not irreducible).
The four reducible cubics that arise are: a conic and a line tangent
to the conic; a conic and a transverse line; three lines passing through
a point; and three lines with no point common to all three. Each of these cases
leads to a number of cases depending on how the points
are placed (such as how many are on each component and whether one or more
is a singular point of the cubic, but also depending on the group law of the cubic).
Analyzing these cases would give a complete result of the Hilbert functions
of the form $h_X$ and $h_{2X}$ for a reduced scheme $X$ consisting of 9 distinct
points of the plane.
\end{rem}

\end{document}